\documentclass[11p,leqno]{amsart}
\usepackage{amsmath}
\usepackage{amsfonts}
\usepackage{amssymb}
\usepackage{graphicx}
\usepackage{color}
\newtheorem{theorem}{Theorem}[section]
\newtheorem*{theorem*}{Theorem}
\newtheorem{lemma}{Lemma}[section]
\newtheorem{corollary}[theorem]{Corollary}
\newtheorem{proposition}{Proposition}[section]

\newtheorem{remark}[theorem]{Remark}

\def\heat{\lf(\frac{\p}{\p t}-\Delta\ri)}

\def\Ric{\text{Ric}}
\def\lf{\left}
\def\ri{\right}

\def\p{\partial}

\def\ddbar{\partial\bar{\partial}}

\def\C{\Bbb C}
\def\R{\Bbb R}

\def\aint{\frac{\ \ }{\ \ }{\hskip -0.4cm}\int}

\def\Ric{\operatorname{Ric}}

\numberwithin{equation}{section}

\begin{document}
\title[A  gap theorem]{\bf An optimal gap theorem}

\author{ Lei Ni}


\thanks{The  author was supported in part by NSF grant DMS-0805834.}


\date{}

\maketitle

\begin{abstract} By solving the Cauchy problem for the Hodge-Laplace heat equation for $d$-closed, positive $(1, 1)$-forms, we prove an optimal gap theorem for K\"ahler manifolds with nonnegative bisectional curvature which asserts that the manifold is flat if the average of the scalar curvature over balls of radius $r$ centered at any fixed point $o$ is a function of $o(r^{-2})$. Furthermore via a  relative monotonicity estimate we obtain  a stronger statement,  namely a `positive mass' type result,  asserting that if $(M, g)$ is not flat, then $\liminf_{r\to \infty} \frac{r^2}{V_o(r)}\int_{B_o( r)}\mathcal{S}(y)\, d\mu(y)>0$ for any $o\in M$.

\end{abstract}

\section{Introduction}

In \cite{NT-jdg} the following gap theorem was proved for K\"ahler manifolds with nonnegative bisectional curvature.
\begin{theorem}\label{nt}
Let $M^m$ ($m=\operatorname{dim}_\C(M)$)  be a complete noncompact K\"ahler
manifold with nonnegative holomorphic bisectional curvature. For any $o\in M$, let  $V_o(r)$ be the volume of the ball $B_o(r)$. Then
$M$ is flat if for some $o\in M$
\begin{equation}\label{2-decay}
\frac{1}{V_o(r)}\int_{B_{o}(r)}\mathcal{S}(y)\, d\mu(y)=o( r^{-2})
\end{equation}
provided   that
\begin{equation}\label{extra}
\liminf_{r\to\infty}\left[\exp\lf(-ar^2\ri)\int_{B_{o}(r)}\mathcal{S}^2(y)d\mu(y)\right]<\infty\end{equation}
for some $a>0$.
Here $\mathcal{S}$ denotes the    scalar curvature  of $M$.
\end{theorem}

 A result of this type was originated  in \cite{MSY}, where it was proved that $M$ is isometric to $\C^m$ under much stronger assumptions that $(M^m, g)$ (with $m\ge 2$) is of maximum volume growth (meaning that $V_o(r)\ge \delta r^{2m}$  for some $\delta>0$) and $\mathcal{S}(x)$ decays pointwisely as $r(x)^{-2-\epsilon}$ for some $\epsilon>0$. A Riemannian version of \cite{MSY} was proved in \cite{GW} shortly afterwards (see also \cite{GW2} for related results).
In  \cite{N98}, Theorem 5.1, with a parabolic method introduced on solving the so-called Poincar\'e-Lelong equation,  the result of \cite{MSY} was improved to the cases covering  manifolds of more general volume growth. Since then there are several further works aiming to prove the optimal result. See for example  \cite{nst}, \cite{CZ}.  In  particular the Ricci flow method was later introduced to the study.  Before this paper, the above Theorem \ref{nt} is the best known result. The extra assumption (\ref{extra}) is related to  the uniqueness of the heat equation solution, which is somewhat natural for the method employed in \cite{NT-jdg}. Since \cite{NT-jdg} it has been a natural question whether or not the extra assumption (\ref{extra}) can be dropped.  In view of the recent examples of H. Wu and F. Y. Zheng \cite{Wu-Zheng} on manifolds with nonnegative bisectional curvature, which include manifolds whose scalar curvature can grow to infinity in  any arbitrarily given manner along any divergent sequence of points, it seems unlikely that  (\ref{extra}) holds automatically on K\"ahler manifolds with nonnegative bisectional curvature.  The main purpose of this paper is to prove, via a completely different  method, the optimal result by removing the extra  assumption (\ref{extra}) in Theorem \ref{nt}.

\begin{theorem}\label{main1}  Without (\ref{extra}), but with the rest assumptions, Theorem \ref{nt} holds.
\end{theorem}

Before we describe the proof we first explain briefly the approaches adapted in the previous works of \cite{MSY}, \cite{N98}, \cite{nst} and \cite{NT-jdg}. A key common ingredient used in the works of  \cite{MSY}, \cite{N98}, \cite{nst} and \cite{NT-jdg} is to solve the so-called Poincar\'e-Lelong equation $\sqrt{-1}\partial \bar{\partial} u=\rho$,  for a given $d$-closed real $(1, 1)$-form $\rho$ and then show that $\operatorname{trace}(\rho)=0$ using (\ref{2-decay}). In \cite{NT-jdg}, the following result was  proved on solving the Poincar\'e-Lelong equation.

\begin{theorem}\label{pl}
Let $M^m$ be a complete K\"ahler manifold
with nonnegative holomorphic bisectional curvature.  Let $\rho$ be
a real $d$-closed $(1,1)$  form with trace $f$. Suppose $f\ge0$ and
$\rho$ satisfies  the following conditions:
\begin{equation}\label{ass1}
 \int_0^\infty  \left(\aint_{B_o(s)}||\rho||(y)d\mu(y)\right) ds<\infty,
\end{equation}
 and
 \begin{equation}\label{extr2}
\liminf_{r\to\infty}\lf[\exp\lf(-ar^2\ri)\cdot\int_{B_o(r)}||\rho||^2(y) d\mu(y)\ri]<\infty
\end{equation}
for some $a>0$.
 Then there is a solution $u$ of
the Poincar\'e-Lelong equation $\sqrt{-1}\ddbar u =\rho$.
Moreover, for any $0<\epsilon<1$, $u$ satisfies
\begin{eqnarray}\label{bound}
  \alpha_1r\int_{2 r}^\infty k(s)ds+\beta_1\int_0^{2r}sk(s)ds
 &\ge& u(x)\ge \beta_3\int_{0}^{2r}sk(s)ds\nonumber \\
&\quad&\quad \quad -\alpha_2r\int_{2r}^\infty k(s)ds-\beta_2\int_0^{\epsilon
r}sk(x,s)ds
\end{eqnarray}
for some positive constants $\alpha_1(m)$, $\alpha_2(m,\epsilon)$
and $\beta_i(m)$, $1\le i\le 3$, where $r=r(x)$. Here  $k(x,s)= \aint_{B_x(s)}f$ and $k(s)=k(o,s)$, where $o\in M$
is a fixed point.
\end{theorem}

The assumption (\ref{extra}) is related to (\ref{extr2}), which in turn due to that a maximum principle for the sub-solutions to   the heat equation,  is needed in the proof of \cite{NT-jdg}.   Assuming that Theorem \ref{pl} can be proved without (\ref{extr2}), then the upper bound estimate  in (\ref{bound}) shows that the solution $u(x)$ of  $\sqrt{-1}\partial\bar{\partial} u=\Ric$ is of $o(\log r(x))$ growth. Now Theorem \ref{main1} follows from the Liouville theorem for plurisubharmonic functions proved in Theorem 0.2 of \cite{NT-jdg}, which asserts that {\it any continuous plurisubharmonic function with upper growth bound  of $o(\log r(x))$ must be a constant}. Since we do not know how to solve the equation without (\ref{extr2}) we do not take this approach here.

 Here we adapt a  different method. The starting point is an alternate argument  of proving the above mentioned Liouville theorem using the monotonicity principle of \cite{Ni-JAMS}. This alternate method makes uses of the asymptotic behavior of the  solution to a parabolic equation to infer geometric/analytic information of the manifolds. This approach via asymptotic study has also succeeded in the recent resolution of the fundamental gap conjecture in \cite{AC}.   The key here is a sharp differential estimate and derived convexity/monotonicity of heat equation deformation of positive $(1,1)$-forms (here we follow the convention of calling nonnegative $(p, p)$-forms positive). Below is an outline of the proof.

First we solve the Cauchy problem:
\begin{eqnarray}\label{cauchy}
\left\{\begin{array}{ll}
\quad &\frac{\partial  }{\partial t} \eta(x, t)+\Delta_{\bar{\partial}} \eta(x, t) =0,\\
\quad& \eta(x,0 )=\Ric(x)\ge 0.\end{array}
\right.
\end{eqnarray}
Here $\Delta_{\bar{\partial}}$ is the Hodge-Laplacian operator, $\Ric(x)$ is the Ricci form of the K\"ahler metric $(M, g)$. Moreover,  we show that there exists  a long time solution $\eta(x, t)$ with $\eta(x, t)\ge0$ on $M \times [0, +\infty)$ such that $\eta(x,t)$ is $d$-closed for any $t$. Let  $u(x, t)=\Lambda \eta\doteqdot g^{i\bar{j}}\eta_{i\bar{j}}$. (Here $\Lambda$ is the contraction with the K\"ahler metric.)  Since $u(x, t)$ is nonnegative and satisfies the heat equation with $u(x, 0)=\mathcal{S}(x)$, it can be expressed in terms of the heat kernel by the uniqueness theorem for nonnegative solutions, Theorem 5.1 of \cite{LY}, keeping in mind that $(M, g)$ has nonnegative Ricci curvature.

 In the second stage of the argument we establish  the large time asymptotics  for $u(x, t)$ using (\ref{2-decay}). The monotonicity of \cite{Ni-JAMS},  derived from a  sharp differential estimate of Li-Yau-Hamilton type (also called differential Harnack estimate) can be applied to $\eta(x, t)$, which particularly implies that $tu(x, t)$ is monotonically nondecreasing  in $t$ for any $x$. This implies Theorem \ref{main1} in the following way:   the assumption (\ref{2-decay})  and the `moment' type estimate, Corollary 3.2 of \cite{N02} implies  that $\lim_{t\to \infty} t u(x_0, t)=0$. Hence the monotonicity and the strong maximum principle  imply that $tu(x, t)\equiv 0$ for any $x$ and  $t>0$. The flatness then follows from $u(x,0)\equiv 0$ which is clear by continuity. Here a key ingredient allowing the application of  the  differential Harnack (also called LI-Yau-Hamilton) estimate to get the needed monotonicity is  that $\eta(x, t)$, the solution obtained in the previous step  is both $d$-closed and positive.

The major technical contribution of this paper is to  solve (\ref{cauchy}), most importantly  obtaining a $d$-closed, positive solution. It is not too difficult  to obtain a $d$-closed solution or to obtain a positive solution alone. See for example Section 11 of \cite{NN} for a construction on obtaining a positive solution. However it is not easy to see why the construction there gives a $d$-closed solution. The difficult part is to obtain a solution satisfying both conditions. To achieve this goal, we have to study the parabolic Hodge-Laplace  problem  on bounded region $\Omega\subset M$ with {\it absolute boundary condition} and prove that the solution obtained is both $d$-closed and positive. The absolute boundary condition is designed to keep the $d$-closeness. It turns out to be a bit subtle to show  that the positivity is preserved. This requires us to extend Hamilton's tensor maximum principle to a much  general setting.

Heuristically, the advantage of the method here over the previous methods of \cite{MSY}, \cite{N98}, \cite{nst}, \cite{NT-jdg} is that by  solving the Cauchy problem for the differential forms, the solution carries more information than  solving a scalar heat equation (or a Poincar\'e-Lelong equation) as in \cite{NT-jdg} and  \cite{MSY, N98, nst}.  Another major extra force of the argument here is  the monotonicity proved in \cite{Ni-JAMS}, derived out of a sharp differential Harnack estimate for the Hermitian-Einstein flow. For our purpose of proving the optimal result we in fact use an improved version of that monotonicity result (in the sense that no growth condition is needed) in a recent joint work \cite{NN}. Hence in the proof here we crucially use two major techniques in the study of Ricci flow, namely the tensor maximum principle and the differential Harnack estimates.  As a consequence of a new relative monotonicity proved in this paper we obtain the following consequence of Theorem \ref{main1} which can be thought as a `positive mass'  result for K\"ahler manifolds of any dimension since the ADM mass in general relativity is just the limit of  scaling-invariant integrals resembling the integration of the scalar curvature over balls in an asymptotically flat manifold.

\begin{corollary}\label{coro-pm} Let $(M, g)$ be a complete noncompact K\"ahler manifold with nonnegative bisectional curvature, which is not totally flat. Then
$$
\liminf_{r\to \infty} \frac{r^2}{V_o( r)}\int_{B_o( r)}\mathcal{S}(y)\, d\mu(y)>0.
$$
\end{corollary}

 We speculate that  $$\lim_{r\to \infty} \frac{r^2}{V_o( r)}\int_{B_o( r)}\mathcal{S}(y)$$ does exist. This certainly is the case for examples constructed in \cite{Wu-Zheng}.

 Finally we note that in \cite{NT-jdg}, a seemingly weaker assumption is used instead of (\ref{2-decay}). The new relative monotonicity shows that in fact it is equivalent to (\ref{2-decay}). The method of this paper effectively proves the gap result and the above `positive mass' type result for any $d$-closed positive $(1,1)$-forms.

\section{A general maximum principle}

 In this section we generalize the maximum principle   for the parabolic systems (of Hamilton's type \cite{H-86}) to the degenerate parabolic systems, with a  mixed type boundary condition.  It turns out that this is what is necessary for the study of the Hodge-Laplacian heat equation deformation of forms in the next section. Although maximum principle and strong maximum principle have been extensively considered for  parabolic PDEs earlier (cf. for example, \cite{Bony},  \cite{Friedman},  \cite{Hill}, \cite{Po}, \cite{RW} and \cite{Wein}) we  could not find the one fitting into  our application. Hence we include the detailed formulation and proof of one such result here. The formulation benefits from the simplification in \cite{BW} on Hamilton's tensor maximum principle. Despite the simplicity of its proof, it seems to include (sometimes maybe with a straight forward modification of argument)  all the previous known ones.

 Consider $V$, a vector bundle over $M$ (of rank $N$), with a fixed metric $\widetilde{h}$, a time-dependent metric connection
$D^{(t)}$. On $M$, assume that  there are, possibly,  time-dependent metrics $g(t)$ with$\nabla^{(t)}$ being the Levi-Civita connection of $g(t)$.
When the meaning is clear we often omit the supscript $^{(t)}$. Recall that for any smooth section $f$ of $V$, $D^2_{X Y}f=D_X D_Y f-D_{\nabla _X Y} f$.
The main concern here is on a smooth section $f(x, t)$ defined over $M \times [0, T]$, and when $f(x, t)$ stays inside a family of sets $C(t)\subset V$. Correspondingly, if we identify the fiber $V_x$ with $\R^N$, consider the following ODE in $\R^N$
\begin{eqnarray}\label{eq:31-1}
\left\{\begin{array}{ll}
\frac{d\,  f}{d t}&=\phi(f),\cr
f(0)&=f_0.\cr
\end{array}
\right.
\end{eqnarray}
Here $\phi: \R^N \to \R^N$ is a locally Lipschitz map and $f(t)$ is a vector valued (in $\R^N$) function of $t$. Given a closed set $X\in \R^N$, we say that {\it the ODE (\ref{eq:31-1}) preserves the  set $X$} (on $[0, T]$) if for any  smooth solution  $f(t)$ with $f(0)\in X$, $f(t)\in X$ for any $t\in [0, T]$. Recall the concept of the {\it tangent cone} of $X$. For any $p\in X$ we define the tangent cone at $p$,  $T^{\mathcal{C}}_pX$ as the collection of vectors $\xi$ satisfying that for any $x_1\in X^c$, the complement of $X$,  with the property that $\operatorname{dist}(x_1, p)=\operatorname{dist}(x_1, X)$,
\begin{equation}\label{eq:31-2}
\langle \xi, p-x_1 \rangle \ge 0.
\end{equation}
Similarly we say that a vector  {\it $\xi$ is tangent to $X$} (and denote all such vectors by $T_pX$) if
\begin{equation}\label{eq:31-3}
\langle \xi, p-x_1 \rangle =0
\end{equation}
for any $x_1\in X^{c}$ with $\operatorname{dist}(x_1, p)=\operatorname{dist}(x_1, X)$. Without any effort one can define similar concepts for a closed set $X$ in a Riemannian manifold. One simply replaces $p-x_1$ by $-\gamma'(0)$, where $\gamma(s)$ is a length minimizing geodesic joining from $p$ to $x_1$.

  \begin{theorem}\label{thm-max} Assume that $M$ is a compact manifold with boundary and  $\phi: V \to V$ is locally Lipschitz. Let $C(t)\subset V$,   $t\in [0, T]$,  be a family of closed subset, depending continuously on $t$. Suppose that for each $t$, $C(t)$ is invariant under parallel transport, fiberwise convex and that the family $\{C(t)\}$ is preserved by the ODE (\ref{eq:31-1}). Consider $f(x, t)$, a family of smooth sections of $V$ on $[0, T]$. Assume  that for any $t>0$ with $\mathcal{D}(t)\doteqdot \max_{y\in M} \operatorname{dist}(f(y, t), C_y(t))>0$, and any $x$ satisfying  $\operatorname{dist}(f(x, t), C_x(t))=\mathcal{D}(t)$, there locally (near $x$) exist vector fields $X_i$ with $1\le i\le k$ and $Y_j$ with $1\le i\le l$, for some nonnegative integers $k, l$, such that
\begin{equation}\label{assumption1}
\frac{\partial f}{\partial t}-\sum_{i=1}^k D^2_{X_i X_i}f-\sum_{j=1}^l D_{Y_j}f -\phi(f)\in T_{v(x, t)}^{\mathcal{C}} C_x(t),\end{equation} where $v(x, t)\in C_x(t)$ is  the vector in $V_x$ with $\operatorname{dist}(f(x, t), v(x, t))=\operatorname{dist}(f(x, t), C_x(t))$. Assume also on    $(x, t)\in  \partial M\times[0, T]$, if  $\operatorname{dist}(f(x, t), C_x(t))= \mathcal{D}(t)>0$, the exterior normal derivative $D_{\nu}f (x, t)\in T_{v(x, t)}^{\mathcal{C}} C_x(t)$, where $v(x, t)\in C_x(t)$ is  as above.  Then $f(x, t)\in C(t)$ for all $t>0$ given $f(x, 0)\in C(0)$.
\end{theorem}

  We say that {\it $C(t)$ depends continuously on $t$} if $C(t)$ as a family of sets parametrized by $t$  is continuous in $t$ with respect to the pointed Hausdorff topology.  The assumption {\it the ODE (\ref{eq:31-1}) preserves the set $C(t)$}  means that for every $x\in M$, if $V_x$ is the fiber of $V$ over $x$, then $C_x(t)=V_x\cap C(t)$ is preserved by (\ref{eq:31-1}).

We remark that our formulation has several advantages. First it applies to the degenerate parabolic systems. Secondly, it applies to some  mixed type boundary value problems. Thirdly it does not require that $f(x, t)$ satisfies a partial differential relation (\ref{assumption1}) everywhere, but only on the extremal points (namely those $x$ with $\operatorname{dist}(f(x, t), C_x(t))=\mathcal{D}(t)$). This, for an example, allows one to modify such a result, by allowing non-smooth boundary,  to include the situations  as those considered in \cite{AC}. When $V$ is a trivial line bundle with $C(t)$ being half line $\{z|\, z\ge c\}$, $\phi(f)$ being zero,  the differential relation (\ref{assumption1}) can be replaced by a degenerate nonlinear one such as
$$
\frac{\partial f}{\partial t}-\det\left(D^2_{X_i X_j} f\right) \ge 0.
$$

The following simple lemma is the basic block for the proof of the  theorem.

\begin{lemma}\label{lemma31}   Let  $\rho(t)$ be a continuous function on $[0, b]$.  Assume  that $\rho(0)\le 0$ and there exist some positive constants $\epsilon, C$ such that   $D^{-}\rho \le C \rho$, whenever $0<\rho(t)\le \epsilon$. Then $\rho(b)\le 0$. The same result holds if $D^{-}$ is replaced by $D^{+}$, $D_{-}$ or $D_{+}$.\end{lemma}

 Recall that the upper left derivative $D^{-}f(t)\doteqdot\limsup_{h\to 0, h>0} \frac{1}{h}\left(f(t)-f(t-h)\right)$.
Similarly, $D^+f(t)\doteqdot\limsup_{h\to 0, h>0}\frac{1}{h}\left(f(t+h)-f(t)\right)$ and $D_{+}f$ and $D_{-}f$ are defined by replacing $\limsup$ by $\liminf$.

 \begin{proof} Let $\widetilde{\rho}(t)=e^{-Ct}\rho(t)$. Then $D^{-} \widetilde{\rho}(t)\le 0$ whenever  $0<\widetilde{\rho}(t)\le \epsilon'$, where $\epsilon'=\epsilon e^{Cb}$. We shall show that, via contradiction,  for any $\eta \le \frac{\epsilon'}{2b}$, $\widetilde{\rho}(t)\le \eta t$. This is sufficient  to prove lemma. Assume the otherwise, then there exist $t_0$ such that $\widetilde{\rho}(t_0)>\eta t_0$. Clearly $t_0>0$ and we may choose $\delta>0$ such that $\widetilde{\rho}(t_0)\ge \eta t_0 +\delta$. Choose $\delta$ so small that $\eta b +\delta <\epsilon'$.  Then let $t_1$ be the smallest $t$ such that $\widetilde{\rho}(t)\ge \eta t+\delta$. Note that $0<\widetilde{\rho}(t_1)=\eta t_1 +\delta<\epsilon'$ and $t_1>0$. On the other hand, by the definition of $D^{-} \widetilde{\rho}(t_1)$, for $h$ small, we have that $\widetilde{\rho}(t_1-h)\ge \widetilde{\rho}(t_1)-\frac{\eta}{2}h\ge \eta t_1+\delta -\frac{\eta}{2}h>\eta(t_1-h)+\delta$.
 \end{proof}

The following proposition, which  is well-known to experts,  answers when the
 ODE (\ref{eq:31-1}) preserves a closed set $X$.  The ODE preserving the set is one of the assumptions in the theorem.

\begin{proposition}\label{prop21} Assume that $X$ is a closed set. Suppose further that $\phi$ is locally Lipschitz on $X_\epsilon$, where $X_\epsilon=\{x\, |\, \operatorname{dist}(x, X)\le\epsilon\}$.      Then (\ref{eq:31-1}) preserves $X$ on $[0, T]$ if and only if $\phi(p)\in T^{\mathcal{C}}_pX$ for any $p\in X$.
\end{proposition}

\begin{proof} We first prove that the condition is sufficient. Let $\rho(t)=\operatorname{dist}^2(f(t), X)$. We shall show that there exists $C>0$ such that, whenever $0<\rho(t)\le \epsilon^2$ ,
$
D^{+} \rho \le C \rho.
$
Since $f([0, T])$ is compact, it is not hard to see that all points $y$ lying on the interval of $\overline{x\, f(t)}$,  where $x$ is over all the points satisfying $\operatorname{dist}(f(t), x)=\operatorname{dist}(f(t), X)$ for the $t$ with  $\rho(t)\le \epsilon^2$,  is contained in a compact subset. Hence, there exist $L>0$ such that
$|\phi(f(t))-\phi(x)|\le L|f(t)-x|$. Now fix $t$ with the property $0<\rho(t)<\epsilon^2$,
and let $x_0$ be the point in $X$ such that $\operatorname{dist}(f(t), X)=\operatorname{dist}(f(t), x_0)$. Compute
\begin{eqnarray*} \limsup_{h\to 0, h>0} \frac{\rho(t+h)-\rho(t)}{h}&\le &\limsup_{h\to 0, h>0} \frac{|f(t+h)-x_0|^2-|f(t)-x_0|^2}{h}
\cr
&=&2\langle \phi(f)(t), f(t)\rangle -2\langle \phi(f)(t), x_0\rangle.
\end{eqnarray*}
Using the assumption that $\langle \phi(x_0), x_0-f(t)\rangle \ge0$ we conclude that
 \begin{eqnarray*} \limsup_{h\to 0, h>0} \frac{\rho(t+h)-\rho(t)}{h}&\le& 2\langle \phi(f)(t), f(t)-x_0\rangle +2\langle \phi(x_0), x_0-f(t)\rangle
\cr
&=&2\langle \phi(f)(t)-\phi(x_0), f(t)-x_0\rangle \cr
&\le& 2L|f(t)-x_0|^2=2L \rho(t).
\end{eqnarray*}
Applying Lemma \ref{lemma31} we have  proved that the assumption $\phi(p)\in T^{\mathcal{C}}_pX$ for any $p$ is sufficient.

In order to see that the condition is necessary we choose a point $x_0$ such that $\phi(x_0)$ does not lie inside the tangent cone $T^{\mathcal{C}}_{x_0}X$. This implies that there exists a $x_1\in X^c$ such that $\operatorname{dist}(x_1, X)=\operatorname{dist}(x_1, x_0)$ and $\langle x_0-x_1, \phi(x_0)\rangle <0$.  Consider a solution $f(t)$ to (\ref{eq:31-1}) with $f(0)=x_0$. Let $\rho(t)=|f(t)-x_1|^2$, then it is easy to see that
\begin{eqnarray*} \limsup_{h\to 0, h>0} \frac{\rho(h)-\rho(0)}{h}&=& 2\langle \phi(f)(0), x_0-x_1\rangle
\cr
&<&0.
\end{eqnarray*}
Hence $\rho(h)<\rho(0)$, which implies that $\operatorname{dist}(x_1, f(h))<\operatorname{dist}(x_1, X)$.  So $f(h)$ is outside $X$ for small $h$.
\end{proof}

\begin{remark}\label{remark32} It is obvious from the proof that if $X$ is time-dependent such that it is decreasing in time in the sense that $X(t')\subset X(t)$ for any $t'\ge t$, the same result holds. \end{remark}

The following lemma can be derived  from \cite{BW}.  We include the argument for the sake of completeness.

\begin{lemma}\label{lemma33} Let $(V, \widetilde{h})$ be a vector bundle with metric $\widetilde{h}$. Let $C\subset V$ be a closed subset such that $C_x=C\cap V_x$ is convex for any $x\in M$. Suppose that $C$ is invariant under the parallel transport with respect to a metric connection $D$ on $V$. If $f(y)$  a section of $V$ satisfies that $f(y)\in C$ for all $y\in M$. Then for any $x\in \operatorname{Int}(M)$, and any $X\in T_xM $ $(D^2_{X X} f)(x)\in T^{\mathcal{C}}_{f(x)}C_{x}$.\end{lemma}

 \begin{proof}   Pick a point $o\in \operatorname{Int}(M)$. Choose a normal coordinate $(x_1,  \cdots, x_n)$ near $o$ such that $\frac{\partial}{\partial x^1}|_{o}=\frac{X}{|X|}$. Assume that $\{E_a\}$ (for $1\le a\le N$ with $N=\operatorname{rank}(V)$) is a basis of $V$ near $o$ so that $E_a$ is parallel along any radial direction. Write $f(x)=\sum \chi_a(x) E_a(x)$. Then since $(x_1, \cdots, x_n)$ is normal, $D^2_{11}f(o)=D_1 D_1 f(o)=\sum_{a=1}^N (\nabla_1 \nabla_1 \chi_a)(o)E_a(o).$ Since $\gamma(s)= (s, 0, \cdots, 0)$ with $-\delta_1\le s\le \delta_1$ for some small $\delta_1>0$  is a short geodesic and if we restrict $f$ on this curve we have $f(s)=\sum \chi_a(s) E_a(s)$ and $(D_1D_1 f)(o)=\sum \chi''_a(0) E_a(o)$. The assumption that $C$ is invariant under the parallel transport amounts to that if we identify $C_o$ with a set $\widetilde{C}\subset \R^N$, via the basis $\{E_a(o)\}$, then the set $C_{\gamma(s)}$ can be identified with the same  $\widetilde{C}$ via the basis $\{E_a(s)\}$. Hence the assumption that $f(x)\in C$ implies that $G(s)=(\chi_1(s), \cdots, \chi_N(s))$ is a path inside $\widetilde{C}$. We shall show that $G''(0)$ lies inside $T^{\mathcal{C}}_{G(0)}\widetilde{C}$. Suppose that $q\in \R^N$ satisfies $\operatorname{dist}(q, G(0))=\operatorname{dist}(q, \widetilde{C})$. By convexity, $\eta G(s)+(1-\eta) G(0)\in \widetilde{C}$ for any $\eta>0$. From the fact that
$
|\eta(G(s)-G(0))+(G(0)-q)|^2 \ge |G(0)-q|^2
$
it is easy to see that $\langle G(s)-G(0), G(0)-q\rangle \ge 0$.
 Hence $I(s):=\langle G(s)-G(0), G(0)-q\rangle\ge 0$ for all $s\in [-\delta_1, \delta_1]$ with $I(0)=0$. This by calculus implies that $\langle G''(0), G(0)-q\rangle \ge 0$. Namely $G''(0)\in T^{\mathcal{C}}_{G(0)}\widetilde{C}$. This shows that $(D_1D_1 f)(o)\in T^{\mathcal{C}}_{f(o)}C_{o}$. \end{proof}

When $M$ is a manifold with boundary we also need the following lemma to handle the boundary points.

\begin{lemma}\label{lemma34} Let $M, (V, \widetilde{h})$, $C$,  and $f(y)$ be as in Lemma \ref{lemma33}. Assume that $M$ is a closed manifold with smooth boundary $\partial M$,  $x\in \partial M$, and at $x$,   $D_{\nu} f(x)\in T^{\mathcal{C}}_{f(x)}C_x$. Here $\nu$ is the exterior unit normal at $x$. Then the same conclusion as Lemma \ref{lemma33} holds.
 \end{lemma}

\begin{proof}  Given and $o\in \partial M$,  choose a coordinate as in the proof of Lemma \ref{lemma33} near $o$ such that $\gamma_1(s)=(s, 0,\cdots, 0)$ is a geodesic from $o$ pointing normally inward and $\gamma_2(s)=(0, s, 0, \cdots, 0), \cdots,  \gamma_n(s)=(0, \cdots, 0, s)$ are just curves lying inside $\partial M$ with $\gamma_i(0)=o$. Moreover $\gamma_1(s)$ only defined for $\delta_1\ge s\ge 0$, while $\gamma_i(s)$ is defined for $-\delta_1\le s\le \delta_1$ with $\delta_1>0$ being small. For $\gamma(s)=\gamma_1(s)$ the same argument as in the proof of Lemma \ref{lemma33} shows that $I(s)=\langle G(s)-G(0), G(0)-q\rangle \ge 0$ with $I(0)=0$. Here $G(s)$, $\widetilde{C}$ and $q$ are as in the proof of Lemma \ref{lemma33}, with $\operatorname{dist}(q, G(0))=\operatorname{dist}(q, \widetilde{C})$. This implies that $I'(0)\ge 0$, namely
$$
\langle G'(0), G(0)-q\rangle \ge 0
$$
which is equivalent to
 $-D_\nu f (o)\in T^{\mathcal{C}}_{f(o)}C_o$. By the assumption that $D_{\nu} f(o)\in T^{\mathcal{C}}_{f(o)}C_o$, we conclude that
$$
\langle -D_\nu f(o), f(o)-q\rangle=0,\quad \mbox{ i. e. } \quad I'(0)=\langle G'(0), G(0)-q\rangle  =0.
$$
This is the same as that $D_\nu f(o)$ is tangential to $C_o$.
Here we abuse the notation by using $q$ to denote both the point $q\in \Bbb{R}^N$ with $\operatorname{dist}(q, G(0))=\operatorname{dist}(q, \widetilde{C})$ and its corresponding point in $V_o$.  Since now the Taylor expansion for $I(s)$ is $\frac{1}{2}I''(0)s^2+O(s^3)\ge 0$ which implies that
$$
I''(0)=\langle G''(0), G(0) -q\rangle \ge 0.
$$
This proves that $D_1D_1 f\in T^{\mathcal{C}}_{f(o)}C_o$, consequently $D^2_{11} f\in T^{\mathcal{C}}_{f(o)}C_o$. When $\gamma(s)=\gamma_i(s)$ for $n\ge i\ge 2$, the same argument as in the proof of Lemma \ref{lemma33} proves that
$$
D_i D_i f (o)\in T^{\mathcal{C}}_{f(o)}C_o.
$$
In order to show that $D^2_{ii} f(o)\in T^{\mathcal{C}}_{f(o)}C_o$, observing that
$$
D^2_{ii} f=D_i D_i f -D_{\nabla_{\gamma'}{\gamma'}} f
$$
it suffices to show that $D_{\nabla_{\gamma'}{\gamma'}} f(o)$ is tangential to $C_o$. By the proof of Lemma \ref{lemma33}, for $i\ge 2$, it is easy to see that
$$
I'(0)=\langle G'(0), G(0)-q\rangle =0
$$
which then implies that $D_{\nabla^{\top} _{\gamma'}{\gamma'}} f(o)$ is tangential to $C_o$. Here $\nabla^{\top} _{\gamma'}{\gamma'}$ and $\nabla^{\bot} _{\gamma'}{\gamma'}$ are the tangential and normal component with respect to $\partial M$.  On the other hand $D_{\nabla^{\bot} _{\gamma'}{\gamma'}}f(o)=-\operatorname{II}(\gamma', \gamma')D_{\nu}f(o)$ is tangential to $C_o$, by the above argument for $i=1$, where $\operatorname{II}(\cdot, \cdot)$ is the second fundamental form of $\partial M $ at $o$. Hence $D^2_{ii} f(o)\in T^{\mathcal{C}}_{f(o)}C_o$.

 After showing that $D^2_{ii}f(o)\in T^{\mathcal{C}}_{f(o)}C_o$ for all $i$ with $1\le i\le n$,  to complete the proof, we only need to show that $D^2_{XX}f(o) \in T^{\mathcal{C}}_{f(o)}C_o$ for $X$ which has both non-zero tangential (to $\partial M$) and normal components. Observing that the above argument has proved that for any $Y\in T_oM$, $D_Y f(o)\in T_{f(o)}C_o$, namely $D_Y f(o)$ is tangential to $C_o$. With the above notations we may assume that $X=a\frac{\partial}{\partial x^1}+b\frac{\partial}{\partial x^2}$ for some $a, b\ne 0$. Since it dose not affect the value of $D^2_{XX}f$ when $X$ is replaced by $-X$ we may also assume that $a>0$. Let $\gamma(s)=(as, bs, 0, \cdots, 0)$ with $s\ge 0$. Then $\gamma'(0)=X$ and $\gamma(s)\in M$. Applying the above argument as above again we show that
 $D_{\gamma'}D_{\gamma'} f(o)\in T^{\mathcal{C}}_{f(o)}C_o$. Hence $D^2_{XX} f(o)\in T^{\mathcal{C}}_{f(o)}C_o$.
\end{proof}

\begin{remark} \label{rk21} The proof of Lemmas \ref{lemma33} and \ref{lemma34} implies that, under either assumption,  for any vector $Y\in T_xM$,
$D_{Y}f(x)\in T_{f(x)}C_x$.
\end{remark}

Now we apply Lemma  \ref{lemma33} and Lemma \ref{lemma34} to prove the theorem.
\begin{proof} Here we follow the simplification of \cite{BW} on the proof of Hamilton's tensor maximum principle. First we  can determine a compact subset $K$ (of V)  containing  an $2\epsilon$-neighborhood of $f(M\times 0, T])$, such that $\phi$ is uniformly Lipschitz on $K$ with Lipschitz constant $L$. We also can assume that $\phi$ is bounded by $L$ on $K$. Let $\rho(x, t)=\operatorname{dist}^2(f(x, t), C_x(t))$. Also let $\rho(t)=\max_{x\in M} \rho(x, t)$. The goal is to apply Lemma \ref{lemma31} to $\rho(t)$. For this purpose we  show that for $t$ satisfying that $0<\rho(t)<\epsilon^2$, there exists $C$ such that $D_{-}\rho \le C\rho$.
Let $x_0$ be the point that $\rho(x_0, t)=\rho(t)$. First we consider the case that $x_0$ lies in the interior of $M$.   For $h$ sufficiently small, let $v(t-h)\in C_{x_0}(t-h)$ be the vector such that $\operatorname{dist}(f(x_0, t-h), C_{x_0}(t-h))=\operatorname{dist}(f(x_0, t-h), v(t-h))$.
By the assumption that (\ref{eq:31-1}) preserves $C(t)$, we infer that
$v(t-h)+h\, \phi(v(t-h))$ is a good approximation to a point in $C_{x_0}(t)$ in the sense that
\begin{equation}\label{eq:32-2}
\operatorname{dist}(v(t-h)+h\, \phi( v(t-h)), C_{x_0}(t))\le C_1 h^2. \end{equation}
This can be seen by considering $\widetilde{f}$, a solution to the ODE (\ref{eq:31-1}) (in the fiber $V_{x_0}$) with $\widetilde{f}(t-h)=v(t-h)$. The assumption that (\ref{eq:31-1}) preserves $C(t)$ implies that $\widetilde{f}(t)\in C_{x_0}(t)$. The claimed result follows from the observation that
\begin{equation}\label{eq:32-3}
|v(t-h)+h\, \phi( v(t-h))-\widetilde{f}(t)|\le C_1 h^2, \end{equation}
where $C_1$ depends on $L$. This can be seen fairly easily. It is also easy to see that there exists subsequence of $h_i\to 0$ such that $v(t-h_i)$ converges to, say $v_\infty$. Abusing the notation, we shall write $h\to 0$  even we really take $h_i\to 0$. Since $\operatorname{dist}(v(t-h), C_{x_0}(t-h))-\operatorname{dist}(C_{x_0}(t-h), C_{x_0}(t))\le \operatorname{dist}(v(t-h), C_{x_0}(t))\le \operatorname{dist}(v(t-h), C_{x_0}(t-h))+\operatorname{dist}(C_{x_0}(t-h), C_{x_0}(t))$, taking $h\to 0$, it is clear that $v_\infty \in C_{x_0}(t)$. Moreover, since that
$\operatorname{dist}(f(x_0, t-h), C_{x_0}(t))-\operatorname{dist}(C_{x_0}(t-h), C_{x_0}(t))\le \operatorname{dist}(f(x_0, t-h), C_{x_0}(t-h))=\operatorname{dist}(v(t-h), f(x_0, t-h))\le \operatorname{dist}(f(x_0, t-h), C_{x_0}(t))+\operatorname{dist}(C_{x_0}(t-h), C_{x_0}(t))$, taking $h\to 0$ we deduce that
 \begin{equation}\label{eq:32-4}\operatorname{dist}(v_\infty, f(x_0, t))=\operatorname{dist}(f(x_0, t), C_{x_0}(t)). \end{equation}
Now we can estimate
\begin{eqnarray*}
\rho(t)-\rho(t-h)&\le& \operatorname{dist}^2(f(x_0, t), C_{x_0}(t))-\operatorname{dist}^2(f(x_0, t-h), C_{x_0}(t-h))\cr
&\le& |f(x_0, t)-\widetilde{f}(t)|^2 -|f(x_0, t-h)-v(t-h)|^2\cr
&\le& |f(x_0, t)-(v(t-h)+h\phi(v(t-h)))|^2 \cr
&\quad& -|f(x_0, t-h)-v(t-h)|^2+O(h^2)\cr
&\le& |f(x_0, t)|^2- |f(x_0, t-h)|^2-2\left(\langle f(x_0, t), v(t-h)+h\phi(v(t-h))\rangle\right.\cr
&\quad& \left. +2\langle f(x_0, t-h), v(t-h)\rangle \right)+2\langle v(t-h), \phi(v(t-h))\rangle h +O(h^2).\cr
\end{eqnarray*}
This shows that
\begin{eqnarray*}
 D_{-}\rho(t)& \le& 2 \langle  \frac{\partial f}{\partial t}, f\rangle|_{(x_0, t)} -2\langle \frac{\partial f}{\partial t}|_{(x_0, t)}, v_\infty\rangle -2\langle f(x_0, t)-v_\infty, \phi(v_\infty)\rangle\cr
 &=& 2\langle (\frac{\partial f}{\partial t} -\sum_{i=1}^k D^2_{X_iX_i} f -\sum_{j=1}^l D_{Y_j}f -\phi(f), f(x_0, t)-v_\infty\rangle \cr
&\quad& +2\langle (\sum_{i=1}^k D^2_{X_i, X_i} f)(x_0, t)+(\sum_{j=1}^l D_{Y_j} f)(x_0, t), f(x_0, t)-v_\infty\rangle\cr
&\quad&  +2\langle \phi(f(x_0, t))-\phi(v_\infty), f(x_0, t)-v_\infty\rangle\cr
&\le& 2\langle (\sum_{i=1}^k D^2_{X_i, X_i} f)(x_0, t)+(\sum_{j=1}^l D_{Y_j} f)(x_0, t), f(x_0, t)-v_\infty\rangle\cr
&\quad&  +2\langle \phi(f(x_0, t))-\phi(v_\infty), f(x_0, t)-v_\infty\rangle.
\end{eqnarray*}
Here we have used the assumed partial differential relation (\ref{assumption1}).
To apply  Lemma \ref{lemma33} to the proof of the theorem,  for any $C$  a convex closed subset set of $\Bbb{R}^N$ let $C_\delta=\{v\, | \operatorname{dist}(v, C)\le \delta\}$. We call $C_\delta$ the $\delta$-neighborhood of $C$. Then $C_\delta$ is also a convex closed subset. Suppose that $p\in C^c$ and  $q\in C$ is a point satisfying $\operatorname{dist}(q, C)=\operatorname{dist}(q, p)$. Choose $\delta=\operatorname{dist}(p, q)$. Clearly $p\in C_\delta$. By the convexity it is easy to see that  $C$ is a subset of the half plane
$H=\{ y\, |\langle y-q, q-p\rangle \ge 0\}$.
 Thus $C_\delta \subset H_\delta =\{ v\, | \langle v-p, q-p\rangle \ge 0\}$. Now by abusing the notation  let $C_\delta$ be the subset of $V$ such that $C_\delta\cap V_x$ is the $\delta$-neighborhood of $C\cap V_x$. Then $f(x, t)\in C_{\sqrt{\rho(t)}}$ due to the choice of $x_0$. Moreover $C_{\sqrt{\rho(t)}}$ is invariant under the parallel transport and $C_{\sqrt{\rho(t)}}\cap V_x$ is convex. Now applying Lemma \ref{lemma33} we conclude that
 $$\langle (\sum_{j=1}^k D^2_{X_j X_j} f)(x_0, t), f(x_0, t)-v_\infty\rangle\le 0.$$
 Observing from Remark \ref{rk21}, $\langle (\sum_{j=1}^l D_{Y_j} f)(x_0, t), f(x_0, t)-v_\infty\rangle=0$,  we then arrive at
\begin{eqnarray*} D_{-}\rho(t)& \le& 2\langle \phi(f(x_0, t))-\phi(v_\infty), f(x_0, t)-v_\infty\rangle\cr
&\le &2L |f(x_0, t)-v_\infty|^2\cr &=&2L \rho(t). \end{eqnarray*}

When $x_0\in \partial M$, we replace Lemma \ref{lemma33} by Lemma \ref{lemma34} to obtain the same  estimate as the above. By Lemma \ref{lemma31}, this completes the proof. \end{proof}

\section{The proof of Theorem \ref{main1}}
Let $m=\dim_{\C}(M)$ and $n=2m$ be the real dimension of $M$.
First choose $\{\Omega_\mu\}$ a sequence of    relatively compact smooth exhaustion domains of $M$. We solve the initial-boundary value problem:
\begin{eqnarray}\label{boundary1}
\left\{\begin{array}{ll}
\quad & \left(\frac{\partial  }{\partial t} +\Delta_{\bar{\partial}}\right) \eta_k(x, t) =0, \quad  \mbox{ on } \,  \Omega_\mu \times [0, \infty), \\
\quad & \operatorname{{\bf n}}\eta_\mu= \operatorname{{\bf n}} d\eta_\mu =0, \mbox {on }\partial \Omega_\mu, \\
\quad& \eta_\mu(x,0 )=\Ric(x), \quad \mbox{on} \, \Omega_\mu.\end{array}
\right.
\end{eqnarray}
Here $\Delta_{\bar{\partial}}=\bar{\partial}\bar{\partial}^* + \bar{\partial}^*\bar{\partial} $, with
$\bar{\partial}^*$ being the adjoint of $\bar{\partial}$, the boundary condition is the so-called {\it absolute boundary condition} with $  \operatorname{{\bf n}} \phi= \iota_{\nu}\phi$, where $\nu$ is the exterior unit normal to $\partial \Omega_\mu$. Recall that $\iota_\nu$ is the adjoint operator of $\nu^* \wedge (\cdot) $.
The solvability follows from  the  theory of linear parabolic systems \cite{Friedman, Lady1}. The solvability for the corresponding elliptic problem can be found for example in Theorem 7.8.4 of
Morrey's classics \cite{Morrey}. Note that the Ricci form  $\Ric(x)=\frac{\sqrt{-1}}{2\pi}R_{i\bar{j}}dz^i\wedge d\bar{z}^{j}$ is a closed $(1,1)$-form. The following result asserts that  the solution $\eta_\mu(x, t)$ will preserve both the positivity and the closeness. For the simplicity of the notation we shall omit the subscripts if the meaning is clear.

\begin{proposition}\label{prop41} Assume that $\Omega$ is a smooth bounded domain. Let $\eta(x, t)$ be the unique solution of the initial-boundary value problem (\ref{boundary1}). Then $\eta(x, t)$ will be a real $(1,1)$-form with $\eta(x, t)\ge 0$ and $\eta(x, t)$ being $d$-closed. The same conclusion holds if $\operatorname{Ric}(x)$ is replaced by any $d$-closed positive real $(1,1)$-form .
\end{proposition}
\begin{proof} Recall an elementary lemma (Lemma 7.5.3 of \cite{Morrey}) which asserts that for any $\alpha, \beta$, smooth $r$ and $r-1$ forms
\begin{equation}\label{eq:41}
\int_\Omega \langle \alpha, d\beta \rangle \, d\mu = \int_\Omega \langle \delta \alpha, \beta\rangle \, d\mu +\int_{\partial \Omega} (-1)^{r-1} \langle \operatorname{{\bf n}}\alpha, \operatorname{\bf{t}} \beta\rangle\, dA
\end{equation}
where $dA$ is the induced surface measure of $\partial M$, $\delta$ is the adjoint of exterior differentiation $d$.   Recall also that for a K\"ahler manifold $2\Delta_{\bar{\partial}}=\Delta_d=d \delta +\delta d$. We shall also write $\eta_\mu$ by $\eta$. The uniqueness of problem (\ref{boundary1}) can be seen via the monotonicity of
$
J(t)\doteqdot \int_{\Omega} |\eta|^2\, d\mu,
$
since
\begin{eqnarray*}
2J'(t)&=&-\int_{\omega} \langle \eta, \Delta \eta\rangle\, d\mu \\
&=&-\int_{\Omega} |d \eta|^2+|\delta \eta|^2\, d\mu.
\end{eqnarray*}
In the above we have used the elementary identity (\ref{eq:41}) twice.
Hence if $J(0)=0$, $J(t)=0$. Applying this observation to the difference of two solutions to (\ref{boundary1}) we have the uniqueness.

The uniqueness  implies that if $\eta(x, 0)$ is a real $(1, 1)$-form, it will be a real $(1, 1)$ form for all $t>0$.
We then proceed to prove the $d$-closeness of $\eta$. Consider
$$
I(t)\doteqdot\int_{\Omega} |d\eta|^2(x, t)\, d\mu(x).
$$
 Recalling that  $\Delta_d=d\delta+ \delta d=2\Delta_{\bar{\partial}}$,  direct calculation shows that
\begin{eqnarray*}
2I'(t)&=&- \int_{\Omega} \langle d \Delta_d \eta, d\eta\rangle(x, t)\,  d\mu(x)\\
&=& -\int_{\Omega} \langle d \, \delta\,  d \eta, d\eta\rangle(x, t)\, d\mu(x)\\
&=& -\int_{\Omega} |\delta d \eta|^2(x, t)\, d\mu(x)\le 0.
\end{eqnarray*}
Here in the first equality we use the fact that $[d, \Delta_d]=0$, in the third equation we used the boundary condition $\operatorname{{\bf n}} d\eta =0$ and identity (\ref{eq:41}), namely Lemma 7.5.3 of \cite{Morrey}. Note that the absolute boundary condition allows one to apply (\ref{eq:41}). Since $I(0)=0$, we conclude that $I(t)\equiv 0$, hence $d\eta=0$ for $t>0$.

For the preservation of the non-negativity, we apply the maximum principle Theorem \ref{thm-max} applying to the degenerate setting. First   consider the unitary frame bundle $\mathcal{U}(M)$ over $M$ which is given by the union of collections of $\mathfrak{f}=\{e_1, \cdots,  e_m\}$, where $\{e_i\}$ being a unitary frame of $T'_pM$ (in a neighborhood of $p$), over all $p\in M$. This is a principle $\mathsf{U}(m)$-bundle over $M$.   Denote its projection map by $\pi$. The Levi-Civita connection now defines a horizontal distribution $\mathcal{H}$ such that for any given $\mathfrak{f}\in \mathcal{U}(M)$, $\mathcal{H}_{\mathfrak{f}}$ is spanned by vector fields $\widetilde{X}_1, \cdots, \widetilde{X}_m$ which are {\it horizontal lift} (cf. \cite{AS}) of $e_1, \cdots, e_m$ in a neighborhood $U$ of $p$ with $\pi(\mathfrak{f})=p$. The {\it vertical vectors} can then be identified with the Lie algebra $\mathfrak{u}(m)$.  We may equip $\mathcal{U}(M)$ with a Riemannian metric so that $\pi$ is a Riemannian submersion with the metric on the fiber being the scalar product of $\mathfrak{u}(m)$ (cf. \cite{AS}). For $\eta(x, t)$, a solution to (\ref{boundary1}) we define a smooth function $v(\mathfrak{f})\doteqdot \eta(\frac{1}{\sqrt{-1}}e_1\wedge \overline{e_1})$ on $\mathsf{U}(M)\times [0, T]$. The Bochner-Kodaira lemma (cf. Lemma 2.1 of \cite{NN}) asserts that at $\mathfrak{f}$
\begin{equation}\label{eq:43}
\frac{\partial v}{\partial t}-\sum_{j=1}^m D^2_{\widetilde{X}_j \widetilde{X}_j} v=\mathcal{KB}(\eta)_{1\bar{1}}
\end{equation}
with
$$
\mathcal{KB}(\eta)_{1\bar{1}}=R_{1\bar{1}k\bar{l}}\eta_{l\bar{k}}-\frac{1}{2}\left(R_{1\bar{k}}\eta_{k \bar{1}}+R_{k \bar{1}}\eta_{1\bar{k}}\right).
$$
Here $R_{i\bar{j}k\bar{l}}$ and $R_{i\bar{j}}$ are the curvature tensor and Ricci tensor of  $(M, g)$ respectively.
 We shall apply Theorem \ref{thm-max} (with $V$ being the trivial rank one bundle) to show that the equation (\ref{eq:43}) is enough to  preserves the nonnegativity of $v$.  To apply Theorem \ref{thm-max} it suffices to check the conditions for extremal points. Now consider $\mathcal{U}(M)$ over $\Omega$ and denote it by $\mathcal{U}(\Omega)$. If for some $t>0$, $v(\mathfrak{f}_0, t)<0$ and $v(\mathfrak{f}_0, t)\le v(\mathfrak{f}, t)$ for all $\mathfrak{f}\in \mathcal{U}(\Omega)$, namely $\mathfrak{f_0}$ is a local extremal point, from the definition of $v$ (in terms of $\eta$), it implies that for some $X\in T'_p(M)$ with $|X|=1$, $\eta(\frac{1}{\sqrt{-1}}X \wedge \overline{X})<0$ and $\eta(\frac{1}{\sqrt{-1}}X \wedge \overline{X})\le \eta (\frac{1}{\sqrt{-1}}Y \wedge \overline{Y})$ for all $Y\in T'_pM$ with $|Y|=1$. Let $\omega$ be the K\"ahler form which can be written as $\sqrt{-1}\sum_{i=1}^m e^*_{i}\wedge e^*_{\bar{i}}$.  Let $-\alpha =\eta(\frac{1}{\sqrt{-1}}X \wedge \overline{X})$. Then $\widetilde{\eta}\doteqdot \eta+\alpha \omega\ge 0$ and $\widetilde{\eta}(\frac{1}{\sqrt{-1}}X \wedge \overline{X})=0$. Now the proof of Proposition 2.2 of \cite{NN} implies that $\mathcal{KB}(\widetilde{\eta})_{X\overline{X}}\ge 0$. But direct calculation shows that $\mathcal{KB}(\widetilde{\eta})=\mathcal{KB}(\eta)$. This implies  that at the extremal point $\mathfrak{f}_0$,
 $$
 \mathcal{KB}(\eta)_{X\overline{X}}\ge 0.
 $$
  This verifies the assumption (\ref{assumption1}).

  Assume that  $\mathfrak{f}_0\in \partial (\mathcal{U}(\Omega))$. Now we verify the the boundary condition. It is easy to see that $p=\pi(\mathfrak{f}_0)$ lies on the boundary of $\Omega$. By a perturbation we may assume that $v(\mathfrak{f}, 0)>0$. Choose a unitary frame $\{e_1, \cdots, e_m\}$ near $p$ so that $e_m =\frac{1}{2}\left(\nu -\sqrt{-1} J \nu\right)$, where $\nu$ is the unit normal vector with respect to the includes Riemannian metric. Let $e_{\bar{j}}=\overline{e_j}$. We can write $\eta=\sqrt{-1}\sum_{i, j=1}^m \eta_{i\bar{j}} e^*_i \wedge e^*_{\bar{j}}$. The boundary condition $\operatorname{{\bf n}}(\eta)=0$ implies that
\begin{equation}\label{bv1}
\eta_{i \bar{m}}=\eta_{m \bar{i}}=\eta_{m\bar{m}}=0
\end{equation}
 for all $1\le i \le m-1$ on $\partial \Omega$.
 Consider the functional
 $$
 \mathcal{I}(\epsilon)\doteqdot \frac{\eta(\frac{1}{\sqrt{-1}} (X+\epsilon Z)\wedge (\overline{X+\epsilon Z}))}{\omega(\frac{1}{\sqrt{-1}} (X+\epsilon Z)\wedge (\overline{X+\epsilon Z}))}.
 $$
The assumption that $v(\mathfrak{f}_0, t)$ is the smallest among all $\mathfrak{f}\in \mathcal{U}(\Omega)$ implies that $\mathcal{I}(0)\le \mathcal{I}(\epsilon)$ for any complex number $\epsilon$. The first variation
$\frac{\partial}{\partial \epsilon} \mathcal{I}(0)=\frac{\partial}{\partial \bar{\epsilon}} \mathcal{I}(0)=0$ then implies that
$$
\eta(\frac{1}{\sqrt{-1}}X \wedge \overline{Z})-\eta(\frac{1}{\sqrt{-1}}X \wedge \overline{X})\omega( \frac{1}{\sqrt{-1}}X \wedge \overline{Z})=0
$$
for any $Z$. By letting $Z=e_m$, equation (\ref{bv1}) and the assumed condition that $\eta( \frac{1}{\sqrt{-1}}X \wedge \overline{X})<0$,  imply that $\omega(\frac{1}{\sqrt{-1}}X \wedge \overline{e_m})=0$. Observing that $\{e_k\}$ is a unitary frame, this shows that $X$ is spanned by $\{e_1, \cdots, e_{m-1}\}$. Without the loss of the generality we assume that $X=e_1$. Since
$$
d\eta =\sqrt{-1}\sum_{i, j, k=1}^m \left(\nabla_{e_k} \eta_{i\bar{j}} e^*_k \wedge e^*_{i}\wedge e^*_{\bar{j}}+\nabla_{e_{\bar{k}}}\eta_{i\bar{j}}e^*_{\bar{k}}\wedge  e^*_{i}\wedge e^*_{\bar{j}}\right).
$$
the assumption $\operatorname{{\bf n}}(d\eta)=0$ on $\partial \Omega$ and  equations (\ref{bv1}), imply that
\begin{equation}\label{eq:45}
\nabla_{\nu} \eta_{i\bar{j}}=0
\end{equation}
for any $i, j$ with $1\le i, j\le m-1$. Here we observe that for $1\le i\le m-1$, $e_i$ are all tangential to $\partial \Omega$. Now let $\widetilde{\nu}$ be the horizontal lift of $\nu$. Clearly $\widetilde{\nu}$ is the unit exterior normal to $\partial \mathcal{U}(\Omega)$. To apply Theorem \ref{thm-max}, we only need to verify that at $\mathfrak{f}_0$, $\frac{\partial v}{\partial \widetilde{\nu}}\ge 0$. Let $\gamma(s)$ (with $s\ge 0$) be a geodesic in the direction of $-\nu$ starting at $p$,  and let $\{e_1(s), \cdots, e_m(s)\}$ be a frame which is parallel along $\gamma(s)$. Let $\widetilde{\gamma}(s)$ be the lifting curve starting from $\mathfrak{f}_0$. Then $-\widetilde{\nu}=\widetilde{\gamma}'(0)$ and
\begin{eqnarray*}
-\frac{\partial v}{\partial \widetilde{\nu}}|_{\mathfrak{f}_0}&=&\frac{d}{ds} \left. \eta\left(\frac{1}{\sqrt{-1}}e_1(s)\wedge e_{\bar{1}}(s)\right)\right|_{s=0}\\
&=& \nabla_\nu \eta _{1 \bar{1}}\\
&=&0
\end{eqnarray*}
by equation (\ref{eq:45}).  Theorem \ref{thm-max} can be applied to have that $v\ge 0$ on $\mathcal{U}(\Omega)\times [0, T]$, hence the nonnegativity of $\eta$. This completes the proof of the proposition.
\end{proof}

\begin{remark} Similar argument proves that the unique solution to problem  (\ref{boundary1}) also preserves  both the $d$-closeness and the positivity for $(p, p)$-forms if the condition $\mathcal{C}_p$ in \cite{NN} is assumed for manifold $(M, g)$.

\end{remark}

 Equipped with the above proposition we are ready to prove Theorem \ref{main1}. First we choose a sequence of smooth domains $\Omega_\mu$ which exhausts $M$. It is not hard to see that the problem (\ref{boundary1}) has a long time solution on $\Omega_\mu \times [0, +\infty)$. Besides appealing to the results from \cite{Lady1}, one way to convince the existence is the monotonicity (of non-increasing) of the energy
$$
\mathcal{J}(t)\doteqdot \int_{\Omega} |\delta \eta_\mu|^2\, d\mu.
$$
 This can be proved similarly as the above by observing that $\operatorname{{\bf n}}(\delta \eta_\mu)=0$, which  follows from $\operatorname{{\bf n}}\eta_\mu =0$ by Lemma 7.5.2 of \cite{Morrey}, and applying (\ref{eq:41}),  namely Lemma 7.5.3 of \cite{Morrey}. Let $\eta_\mu$ be the solution on $\Omega_\mu \times [0, \infty)$ to the boundary value problem (\ref{boundary1}). By Proposition \ref{prop41}, we have that $\eta_\mu(x, t)$ is positive and it is $d$-closed. Let $u_\mu(x, t)=\Lambda \eta_\mu(x, t)$. By the positivity of $\eta_\mu$ the estimate of $|\eta_\mu|$ can be reduced to the upper estimate of $u_\mu$. By the identities
 $$
 \partial \Lambda -\Lambda \partial=-\sqrt{-1}\bar{\partial}^*,\quad \quad \bar{\partial} \Lambda -\Lambda \bar{\partial}=\sqrt{-1}\partial^*
 $$
 where $\partial^*$ and $\bar{\partial}^*$ are conjugate operators of $\partial$ and $\bar{\partial}$, and (8.1.19) from \cite{Morrey} which asserts that
 $$
 \iota_\nu \bar{\partial}^* \eta_\mu=0
 $$
 the closeness of $\eta$ implies that
 $$
 \iota_\nu \partial u_\mu =\iota_\nu \bar{\partial} u_\mu =0.
 $$
 This in particular implies that $u_\mu$ satisfies the Neumann boundary condition. Let $H_\mu(x, y, t)$ be the Neumann fundamental solution on $\Omega_\mu$. By the local gradient estimate of Li-Yau we know that, by passing to a subsequence if necessary,
 $H_\mu(x, y, t)$ converges, uniformly on any given compact subset, to a positive solution starting with the $\delta_x(y)$. The uniqueness of the positive solution (see for example Theorem 5.1 of \cite{LY}) implies that the limit is the minimum positive fundamental solution, namely the heat kernel $H(x, y, t)$ on $M$. The assumption on the scalar curvature in Theorem \ref{main1} implies that
 \begin{equation}\label{eq:int1}
 \int_M |\operatorname{Ric}|(y)\exp( -a r^{2-\delta}(y))\, d\mu(y) < \infty
 \end{equation}
 for some $a, \delta>0$. Hence by the heat kernel estimate of Li-Yau:
 \begin{equation}\label{hk-ly}
 \frac{C_1(n)}{V_x(\sqrt{t})} \exp\left(-\frac{r^2(x, y)}{3t}\right)\le H(x, y, t)\le \frac{C_2(n)}{V_x(\sqrt{t})}\exp\left(-\frac{r^2(x, y)}{5t}\right)
 \end{equation}
 with $C_1(n), C_2(n)>0$, out of (\ref{eq:int1}) we conclude the finiteness of
 $$
 u(x, t)\doteqdot\int_M H(x, y, t)\mathcal{S}(y)\, d\mu(y)
 $$
 for any $t>0$.
The local gradient estimate of \cite{LY} also gives  similar upper estimates on $H_\mu(x, y, t)$, hence $u_\mu(x, t)$ as the above ones for $H(x, y, t)$ and $u(x, t)$. Another way to obtain the maximum modulus estimate on $u_\mu$ is to make use of  $k(o, r)\doteqdot \aint_{B_o(r)}\mathcal{S}(y)\, d\mu(y)$ is uniformly (in terms of $r$) bounded from above, which is implied by the assumption (\ref{2-decay}) made in  the theorem. First we can choose $r_\mu, \Omega_\mu$ so that $B_o(r_\mu)\subset \Omega_\mu \subset B_o(2r_\mu)$. The Neumann boundary condition ensures that $\int_{\Omega_\mu} u_\mu(x, t) d\mu(x)$ is a conserved quantity (in $t$). Hence it is
bounded from the above $\int_{B_o(2r_\mu)}\mathcal{S}(y)\, d\mu(y)$. The volume doubling property of the manifold  then implies that there exists an absolute  constant $C$ such that
$$
\aint_{B_o(r_\mu)} u_\mu(y, t)\, d\mu(y)\le C k(o, 2r_\mu).
$$
Then the interior maximum estimate on $u_\mu(x, t)$ now follows from the parabolic mean value inequality (cf. Theorem 1.1 of \cite{LT}).
Now the interior Schauder's estimates (cf. Theorem 6.2.6 of \cite{Morrey}, or \cite{Simon}) can be applied to extract convergent subsequence and obtain a solution $\eta(x, t)$ solving
$$
\left(\frac{\partial}{\partial t}+\Delta_{\bar{\partial}}\right) \eta (x, t)=0
$$
on $M \times [0, \infty)$. It is easy to see that $\eta(x, t)\ge 0$ and it is $d$-closed. Moreover, the positivity of $\eta(x, t)$ implies that $u(x,t)=\Lambda \eta(x, t)$, by the uniqueness. Clearly it bounds $|\eta|$ after multiplying some positive constant depending on $m$. Note that $u(x, t)$ also satisfies the Harnack estimate (cf. \cite{LY})
\begin{equation}\label{eq:add1}
u(x, t)\le u(o, T) \left(\frac{T}{t}\right)^{n/2} \exp\left(\frac{r^2(o, x)}{4(T-t)}\right)
\end{equation}
for any $t<T$. Here $o\in M$ is a fixed point.
Estimate (\ref{eq:add1})  allows one to use  Corollary 2.1 of \cite{Ni-JAMS} via a perturbation argument.  Alternatively  we  may apply  Theorem 4.1 of \cite{NN}, to conclude that
\begin{equation}\label{eq:lyh}
\frac{1}{t}\frac{\partial}{\partial t} \left(t u(x, t)\right)+\langle \partial u, X\rangle +\langle \overline{\partial} u, \overline{X}\rangle +\eta\left(\frac{1}{\sqrt{-1}}X \wedge \overline{X}\right)\ge 0
\end{equation}
for any $(1, 0)$-type vector field $X$. From Theorem 4.1 of \cite{NN} to the above estimate one needs to use that $\eta$ is closed. See Corollary 4.2 of \cite{NN} for details on this.

By taking $X=0$ in (\ref{eq:lyh}), we have that $tu(x, t)$ is monotone non-decreasing. Since $u(x, t)$ is a solution to the heat equation with $u(x, 0)=\mathcal{S}(x)$. Now we can evoke the `moment estimates' in \cite{N02}.
\begin{lemma}[Theorem 3.1, \cite{N02}] Let $u(x, t)$ be the unique nonnegative solution to the heat equation on $M$, with nonnegative Ricci curvature. Assume that $u(x, 0)=f(x)$. If
$\frac{1}{V_x(r)}\int_{B_x(r)} f(y)\, d\mu(y)\le Ar^{-d}$ for some $A>0$, $d\ge -n-2$,  and all $r\ge R$, then $u(x, t)\le C(n, d)At^{d/2}$ for $t\ge R^2$.
\end{lemma}
 Applying the above result to $d=-2$ we have that  $u(o, t)=o(t^{-1})$, from the assumption (\ref{2-decay}) and that $(M,g)$ has nonnegative Ricci curvature. Together with that $tu(o, t)$ is monotone nondecreasing, it implies that
$u(o, t)=0$ for all $t$. The strong maximum principle implies that $u(y, t)\equiv 0$ noting that $u(y, t)\ge 0$ and it solves the heat equation.  Hence $u(y, 0)=S(y)\equiv 0$. This proves the flatness of $(M, g)$ since $(M, g)$ is assumed to have nonnegative bisectional curvature.

\begin{remark} The monotonicity $tu(o, t)$ is the same as the monotonicity of 
$$
t\int_M H(o, y, t)\mathcal{S}(y)\, d\mu(y).
$$
One can  shows that, if denote by $\mathcal{S}_k$ the $k$-th elementary symmetric  function of $\eta(y, 0)$, which is viewed as a Hermitian symmetric bilinear form,
\begin{equation}\label{sigma-k}
t^k\int_M H(o, y, t) \mathcal{S}_k(y)\, d\mu(y)
\end{equation}
is monotone non-decreasing.
\end{remark}

Since we do not make uses of any special features of Ricci form, the result holds for any $d$-closed real positive $(1, 1)$ form $\rho$. More precisely we have proved the following result.

\begin{theorem} Assume that $(M, g)$ is a complete K\"ahler manifold with nonnegative bisectional curvature. Suppose that $\rho=\sqrt{-1}\sum \rho_{i\bar{j}} dz^i \wedge dz^{\bar{j}}$, a real $d$-closed positive $(1,1)$-form. Then $\rho\equiv 0$, if (\ref{2-decay}) holds for some $o\in M$ with $\mathcal{S}(y)=\Lambda \rho (y)$.
\end{theorem}

\section{A relative monotonicity}
In \cite{NT-jdg}, for Theorem \ref{nt}, in stead of (\ref{2-decay}), the following seemingly weaker condition
\begin{equation}\label{3-decay}
\int_0^r s\aint_{B_{o}(s)}\mathcal{S}(y)\, d\mu(y)\, ds =o(\log r)
\end{equation}
is assumed. It is not hard to see that (\ref{2-decay}) implies (\ref{3-decay}). In this section we shall show that these two conditions are equivalent as a result of a general monotonicity estimate. It is not hard to see that (\ref{3-decay}) implies that $\aint_{B_{o}(r)}\mathcal{S}(y)\, d\mu(y)=o\left(\frac{\log r}{r^2}\right)$. But it seems not an easy task to improve it to $o\left(\frac{1}{r^2}\right)$ without extra considerations. The key ingredient in proving that (\ref{3-decay}) implies (\ref{2-decay}) is the following relative monotonicity.

\begin{theorem}\label{thm31} Assume that $(M^m, g)$ is a complete K\"ahler manifold with nonnegative bisectional curvature. Let $\rho$ be a $d$-closed real positive $(1,1)$ form. Then there exists $\delta(m)=\sqrt{\frac{1}{2m}}$ and $C=C(m)>0$ such that
\begin{equation}\label{mono31}
r_1^2\aint_{B_o(r_1)} \mathcal{S}(y)\, d\mu(y)\le C r^2 \aint_{B_o(r)}\mathcal{S}(y)\, d\mu(y)
\end{equation}
for any positive $r_1\le \delta r$. Here $\mathcal{S}=\Lambda \rho$.
\end{theorem}
\begin{proof}  Given any $t_1\in \R$, let
$$
\varphi_{(o, t_1), h}(y, t)=\left(1-\frac{r^2(o, y)+(m-1)(t-t_1)}{h^2}\right)_+
$$
where $r(x, y)$ is the distance function between $x$ and $y$, $h>0$ is any given positive number, $(f)_+$ denotes the positive part of any given continuous function $f$.
Consider the quantity
$$
\mathcal{E}_{t_0, t_1, o, h}(t)=(t_0-t)\int_M \varphi_{(o, t_1), h}H(o, y, t_0-t) \mathcal{S}(y)\, d\mu(y).
$$
When the meaning is clear we simply denote it as $\mathcal{E}(t)$. The key step for the proof is to establish that $\mathcal{E}(t)$ is monotone non-increasing. We  perform the calculation inside the support of $\varphi$, pretending that $\varphi$ is smooth and address how to resolve this issue later. Denote $t_0-t$ by $\tau$. Also abbreviate $\varphi_{(o, t_1), h}$ by $\varphi$.

Direct calculation shows
\begin{eqnarray}
\frac{d}{dt} \mathcal{E}(t)&=& \int_M -\varphi H \mathcal{S} -\tau (\Delta H) \mathcal{S} \varphi +\tau H S \varphi_t  \nonumber\\
&=& \int_M \tau\left(\heat \varphi\right) H \mathcal{S} -H\mathcal{S}\varphi +\tau \langle \nabla' H, \nabla' \mathcal{S}\rangle \varphi -\tau \langle \nabla'  \mathcal{S}, \nabla' \varphi \rangle H. \label{est-41}
\end{eqnarray}
Here $\nabla'f=\nabla_i f\frac{\partial}{\partial z^i}$ (and $\nabla'' f= \overline{\nabla ' f}$).
Using the basic estimate of \cite{CN}:
\begin{equation}\label{hk-est1}
\nabla_i\nabla_{\bar{j}} \log H +\frac{1}{\tau} g_{i\bar{j}} \ge 0
\end{equation}
which implies the complex Hessian estimate for $r^2(o, y)$ (as a function of $y$)
\begin{equation}\label{hessian-com}
\nabla_i \nabla_{\bar{j}} r^2 \le g_{i\bar{j}},
\end{equation}
we have that, under a normal coordinate,
\begin{eqnarray}
0&\le& \int_M \left( \nabla_i\nabla_{\bar{j}} \log H +\frac{1}{\tau} g_{i\bar{j}} \right) H \rho_{i\bar{j}}\, \varphi \nonumber\\
&=& \int_M \frac{1}{\tau} H \mathcal{S} \varphi - \langle \nabla' H, \nabla' \mathcal{S}\rangle \varphi -\rho(\nabla' H, \nabla'' \varphi) -\rho(\nabla' \log H, \nabla'' \log H)H \, \varphi. \label{est-42}
\end{eqnarray}
Here  $\rho(X, \overline{Y})=\rho_{i\bar{j}}X^i Y^{\bar{j}}$ for $X=X^i\frac{\partial}{\partial z^i}, \overline{Y}=Y^{\bar{j}}\frac{\partial}{\partial z^{\bar{j}}}$, and we have used that $\rho$ is $d$-closed, which implies that $\nabla_k \rho_{i\bar{j}}=\nabla_i \rho_{k \bar{j}}$. Integration by parts on the third term of the right hand side  shows that
\begin{equation}\label{equ-41}
-\int_M \rho(\nabla' H, \nabla'' \varphi)=\int_M H \langle \nabla' \varphi, \nabla' \mathcal{S}\rangle +H \rho_{i\bar{j}} \varphi_{j \bar{i}}.
\end{equation}
Combining (\ref{est-42}) and (\ref{equ-41}), also noting that $\rho\ge 0$, we have that
\begin{equation}\label{est-43}
0\le \int_M \frac{1}{\tau} H \mathcal{S} \varphi- \langle \nabla' H, \nabla' \mathcal{S}\rangle \varphi+\langle \nabla' \varphi, \nabla' \mathcal{S}\rangle +H \rho_{i\bar{j}} \varphi_{j \bar{i}}.
\end{equation}
Applying this estimate to (\ref{est-41}), it implies that
\begin{equation}\label{est-44}
\frac{d}{dt}\mathcal{E}(t) \le \int_M \tau\left(\heat \varphi\right) H \mathcal{S} +\tau H \rho_{i\bar{j}} \varphi_{j \bar{i}}.
\end{equation}

Meanwhile, under a normal coordinate which diagonalizes  $\rho_{i\bar{j}}$,
\begin{eqnarray*}
\mathcal{S}\heat \varphi +\rho_{i\bar{j}}\varphi_{j \bar{i}}&=& \frac{1}{h^2}\left(-(m-1)\mathcal{S}+ \nabla_{i}\nabla_{\bar{i}}r^2 \mathcal{S}-\nabla_i\nabla_{\bar{j}} r^2 \rho_{j\bar{i}}\right)\\
&=& \frac{1}{h^2}\left(-(m-1)\mathcal{S}+\nabla_{i}\nabla_{\bar{i}}r^2 \left( \mathcal{S}-\rho_{i\bar{i}}\right)\right)\\
&\le& \frac{1}{h^2}\left(-(m-1)\mathcal{S} +\sum_{i=1}^m ( \mathcal{S}-\rho_{i\bar{i}})\right)\\
&=& 0.
\end{eqnarray*}
In the above inequality, we have used (\ref{hessian-com}) and $\mathcal{S}-\rho_{i\bar{i}}\ge 0$.
This estimate together with (\ref{est-44}) implies that $\frac{d}{dt}\mathcal{E}(t)\le 0$. The non-smoothness of $\varphi$ does not affect the monotonicity since the estimate $\mathcal{S}\heat \varphi +\rho_{i\bar{j}}\varphi_{j \bar{i}}\le 0$ holds in the distribution sense. One can follow the regularization process in Section 5 of  \cite{N-jdg}  (see also  \cite{Ecker} ) by multiplying a cut-off function $\eta_\epsilon$ which is given by $\zeta_\epsilon (\varphi)$ with $\zeta_\epsilon$ being a smooth cut-off function on $\R^1$ satisfying $0\le \zeta_\epsilon(z) \le 1$, $\zeta_\epsilon (z)=1$ for $z\ge -n \log(1-\epsilon)$ and $\zeta_\epsilon(z)=0$ for $z\le 0$. Furthermore $\zeta_\epsilon (z)$ is chosen to satisfies that  $|\zeta_\epsilon'|\le \frac{2}{-n\log(1-\epsilon)}$, $|\zeta_\epsilon'(z) z|\le 2$ for $z\in [0, -n\log (1-\epsilon)]$ and $\zeta'(z)=0$ for other $z$. This ensures that the error terms
$$
\int_M S \varphi H \frac{\partial \eta_\epsilon }{\partial t }, \quad \quad \int_M \varphi |\nabla (H S)||\nabla \eta_\epsilon|
$$
all tend to zero as $\epsilon \to 0$.

Now fixing  $T>0$,  $h>0$ and $h_1\in (0, \delta h)$, where $\delta$ is a positive constant depending only on the dimension $m$ which shall be specified later,  let $t_0=T+h_1^2$, $t_1=T-\delta^2 h^2$. Then we have
$$
\mathcal{E}(T-\delta^2 h^2)\ge \mathcal{E}(T-h_1^2)
$$
by the monotonicity on $\mathcal{E}(t)$. By Li-Yau's heat kernel upper bound (\ref{hk-ly}) we have that
$$
H(o, y, T-\delta^2 h^2)\le \frac{C(m)}{V_o(h)}
$$
for some positive constant $C=C(m).$ Note that $\varphi(t) \le 1$, easy estimates shows that
\begin{equation}\label{est-45}
\mathcal{E}(T-\delta^2 h^2) \le \frac{C(m)\delta^2 h^2}{V_o(h)}\int_{B_o( h)} \mathcal{S}(y)\, d\mu(y).
\end{equation}
On the other hand, for $y\in B_o(h_1)$, we have that
\begin{eqnarray}
\varphi(y, T-h_1^2)&\ge& 1-\frac{h_1^2+(m-1)\delta^2 h^2}{h^2}\nonumber\\
&\ge& 1-m \delta^2\ge 1/2 \label{est-46}
\end{eqnarray}
if we choose $\delta =\sqrt{\frac{1}{2m}}$. On the other hand, the lower estimate on the heat kernel from (\ref{hk-ly}) implies that for $y\in B_0( h_1)$, there exists $C'(m)\ge 0$ such that when $t=T-h_1^2$,
\begin{equation}\label{est-47}
H(o, y, t_0-t)\ge \frac{C'(m)}{V_o(\sqrt{2}h_1)}\ge\frac{C'(m)}{2^{m} V_o(h_1)}.
\end{equation}
Combining (\ref{est-46}) and (\ref{est-47}) we have the following lower estimate:
\begin{equation}\label{est-48}
\mathcal{E}(T-h_1^2)\ge \frac{C''(m) h_1^2}{V_o(h_1)}\int_{B_o(h_1)}\mathcal{S}(y)\, d\mu(y).
\end{equation}

Combining (\ref{est-45}), (\ref{est-48}) and the monotonicity of $\mathcal{E}(t)$, this completes the proof of the theorem.
\end{proof}

\begin{remark} A similar, but more involved, computation also shows that, for all $r_1\le \delta r$, 
\begin{equation}\label{re-sigma-k}
\frac{r_1^{2k}}{V_o(r_1)}\int_{B_o(r_1)}\mathcal{S}_k(y)\, d\mu(y) \le C \frac{r^{2k}}{V_o(r)}\int_{B_o(r)}\mathcal{S}_k(y)\, d\mu(y).
\end{equation}
\end{remark}

Now the implication of (\ref{2-decay}) out of  (\ref{3-decay}), as well as  Corollary \ref{coro-pm} out of Theorem \ref{main1}, can be proved by applying  Theorem \ref{thm31} and some  elementary considerations. As in the last section, our argument effectively proves the following result for any positive $d$-closed, $(1,1)$-form.

\begin{corollary}
Assume that $(M, g)$ is a complete K\"ahler manifold with nonnegative bisectional curvature. Suppose that $\rho=\sqrt{-1}\sum \rho_{i\bar{j}} dz^i \wedge dz^{\bar{j}}$, a real $d$-closed positive $(1,1)$-form with $\mathcal{S}=\Lambda \rho$. Then
$$
\liminf_{r\to \infty} \frac{r^2}{V_o(r)}\int_{B_o(r)}\mathcal{S}(y)\, d\mu(y)>0
$$
unless
$\rho\equiv 0$.
\end{corollary}

\bibliographystyle{amsalpha}

\begin{thebibliography}{A}


\bibitem[AC]{AC} B. Andrews and J. Clutterbuck, \textit{ Proof of the fundamental gap conjecture.} To appear in Jour. Amer. Math. Soc.

\bibitem[AS]{AS} W. Ambrose and I. M.  Singer, \textit{  A theorem on holonomy.}  Trans. Amer. Math. Soc. \textbf{75}(1953), 428--443.

\bibitem[BW]{BW}  C. B\"ohm and B.   Wilking, \textit{  Nonnegatively curved manifolds with finite fundamental groups admit metrics with positive Ricci curvature.}  Geom. Funct. Anal.  \textbf{17}(2007),  no. 3, 665--681.

\bibitem[B]{Bony} J. M. Bony, \textit{ Principe du maximum, in\'galite de Harnack et unicit\'e  du probl\`eme de Cauchy pour les op\'erateurs elliptiques d\'eg\'en\'er\'es.} Ann. Inst. Fourier (Grenoble), \textbf{19}:1 (1969), 277–-304.



\bibitem[CN]{CN} H.-D. Cao and L.  Ni, \textit{ Matrix Li-Yau-Hamilton estimates for the heat equation on K\"ahler manifolds.}  Math. Ann.  \textbf{331}(2005),  no. 4, 795–-807.

\bibitem[CZ]{CZ} B.-L. Chen and X.-P.  Zhu, \textit{ On complete noncompact K\"ahler manifolds with positive bisectional curvature.}  Math. Ann.  \textbf{327}  (2003),  no. 1, 1--23.


\bibitem[E]{Ecker} K. Ecker,  \textit{ Regularity theory for mean curvature flow.} Progress in Nonlinear Differential Equations and their Applications, 57. Birkh\"auser Boston, Inc., Boston, MA, 2004.


\bibitem[F]{Friedman} A. Friedman, \textit{ Partial differential equations of parabolic type.} Prentice Hall, 1964.

\bibitem[GW1]{GW2} R.-E. Greene and H. Wu, \textit{Function Theory on Manifolds Which Process a Pole.} Springer-Verlag, 1979.

\bibitem[GW2]{GW} R.-E. Greene and H. Wu, \textit{ Gap theorems for noncompact Riemannian manifolds.}  Duke Math. J.  \textbf{49}(1982), no. 3, 731–-756.


\bibitem[H]{H-86} R. Hamilton, \textit{Four-manifolds with positive curvature operator.} J. Differential Geom. \textbf{24}(1986), 153--179.

\bibitem[Hi]{Hill} C. D.
Hill, \textit{ A sharp maximum principle for degenerate elliptic-parabolic equations.}  Indiana Univ. Math. J.  \textbf{20}(1970/1971),  213–229.


\bibitem[LU]{Lady1} O.-A.
Ladyzenskaja and N.-N.  Ural'ceva, \textit{
A boundary-value problem for linear and quasi-linear parabolic equations. I, II, III. (Russian)}
Iaz. Akad. Nauk SSSR Ser. Mat. \textbf{26} (1962), 5-52; ibid. \textbf{26} (1962), 753--780; ibid. \textbf{27} 1962 161--240.

\bibitem[LT]{LT} P. Li and L.-F. Tam, \textit{ The heat equation and harmonic maps of complete manifolds.} Invent. Math. \textbf{105}(1991), 1--46.


\bibitem[LY]{LY} P. Li and S.-T. Yau, \textit{ On the parabolic
kernel of the Schr\"odinger operator}. Acta Math. \textbf{156}
(1986),  no. 3-4, 153--201.


\bibitem[MSY]{MSY} N. Mok, Y.-T. Siu ad S.-T. Yau, \textit{ The
Poincar\'e-Lelong equation on complete K\"ahler manifolds.}
Compositio Math. \textbf{44}(1981), 183--218.

\bibitem[Mo]{Morrey}
C. Morrey, \textit{ Multiple Integrals in Calculus of Variations.} Springer-Verlag, New York, 1966.

\bibitem[N1]{N98} L. Ni, \textit{ Vanishing theorems on complete K\"ahler manifolds and their applications.}  J. Differential Geom.  \textbf{50}(1998),  no. 1, 89--122.

\bibitem[N2]{N02} L. Ni, \textit{ The Poisson equation and Hermitian-Einstein metrics on holomorphic vector bundles over complete noncompact K\"ahler manifolds.} Indiana Univ. Math. J. \textbf{51}(2002), 679--704.

\bibitem[N3]{Ni-JAMS} L. Ni, \textit{
A monotonicity formula on complete K\"ahler manifolds with nonnegative bisectional curvature. } Jour. Amer. Math. Soc.  \textbf{17}  (2004),  no. 4, 909--946 (electronic).

\bibitem[N4]{N-jdg} L. Ni, \textit{
A matrix Li-Yau-Hamilton estimate for K\"ahler-Ricci flow.}  J. Differential Geom.  \textbf{75}(2007),  no. 2, 303–-358.

\bibitem[NN]{NN} L. Ni and Y. Y. Niu, \textit{ Sharp differential estimates of Li-Yau-Hamilton type for positive $(p,p)$-forms on K\"ahler manifolds}.    Comm. Pure Appl. Math., to appear.


\bibitem[NST]{nst} L. Ni, Y. Shi and L.-F. Tam, \textit{ Poisson equation, Poincar\'e-Lelong equation and curvature decay on complete K\"ahler manifolds.}  J. Differential Geom.  \textbf{57}  (2001),  no. 2, 339--388.




\bibitem[NT]{NT-jdg} L. Ni and L.-F.  Tam,  \textit{ Plurisubharmonic functions and the structure of complete K\"ahler manifolds with nonnegative curvature.}  J. Differential Geom.  \textbf{64}  (2003),  no. 3, 457--524.





\bibitem[P]{Po} A. Pulemotov, \textit{The Hopf boundary point lemma for vector bundle sections.}  Comment. Math. Helv.  \textbf{83}(2008),  no. 2, 407–-419.


\bibitem[RW]{RW}    R. Redheffer and W. Walter, \textit{Invariant sets for systems of partial differential equations I. Parabolic equations.}  Arch. Rational Mech. Anal. \textbf{67}(1978), 41–-52.

\bibitem[Sm]{Simon}  Simon, L. Schauder estimates by scaling. \textit{ Calc. Var. Partial Differential Equations}  \textbf{5}  (1997),  no. 5, 391--407.


\bibitem[W]{Wein} H. Weinberger, \textit{Invariant sets for weakly coupled parabolic and elliptic systems.} Rend. Mat. (6) \textbf{8}(1975), 295--310.

    \bibitem[WZ]{Wu-Zheng} H. Wu and F.-Y. Zheng, \textit{ Examples of positively curved complete K\"ahler manifolds.} Advanced Lectures in Mathematics, \textbf{17}(2010), Geometry and Analysis Vol I, 517--542.

\end{thebibliography}

{\sc  Addresses:}

{\sc  Lei Ni},
 Department of Mathematics, University of California at San Diego, La Jolla, CA 92093, USA


email: lni@math.ucsd.edu

\end{document}